\documentclass[12pt,a4paper]{amsart}

\newtheorem{theo+}              {Theorem}           [section]
\newtheorem{prop+}  [theo+]     {Proposition}
\newtheorem{coro+}  [theo+]     {Corollary}
\newtheorem{lemm+}  [theo+]     {Lemma}
\newtheorem{exam+}  [theo+]     {Example}
\newtheorem{rema+}  [theo+]     {Remark}
\newtheorem{defi+}  [theo+]     {Definition}
\def \r{\mbox{${\mathbb R}$}}

\newenvironment{theorem}{\begin{theo+}}{\end{theo+}}
\newenvironment{proposition}{\begin{prop+}}{\end{prop+}}
\newenvironment{corollary}{\begin{coro+}}{\end{coro+}}

\usepackage{amsthm}
\theoremstyle{plain} \theoremstyle{remark}
\newtheorem{remark}{Remark}
\newtheorem{example}{Example}

\def\E{/\kern-1.0em \equiv }

\author{}

\begin{document}
\title[Biharmonic  Riemannian submersions]{A short survey on biharmonic  Riemannian submersions}
\subjclass{58E20} \keywords{Biharmonic maps, biharmonic Riemannian submersions, BCV spaces, Thurston's 3-dimensional geometries.}
$^{*}$ Dedicated to Professor B.-Y. Chen for his 80th birthdate\\

\author{Ye-Lin Ou $^{*}$}
\address{Department of
Mathematics,\newline\indent Texas A $\&$ M University-Commerce,
\newline\indent Commerce, TX 75429,\newline\indent USA.\newline\indent
E-mail:yelin.ou@tamuc.edu }
\date{04/06/2024}
\maketitle
\section*{Abstract}
\begin{quote} The study of biharmonic submanifolds, initiated by B. Y. Chen \cite{Ch1, CI} and  G. Y. Jiang \cite{Ji86, Ji87} independently, has received a great attention  in the past 30 years with many important progress (see \cite{On2},  \cite{OC}  and the vast references therein). This note  attempts to give a short survey on the study of biharmonic Riemannian submersions which are a dual concept of biharmonic submanifolds (i.e., biharmonic isometric immersions).
{\footnotesize } 
\end{quote}

\section{Why biharmonic Riemannian submersions?}

\indent A map $\varphi: (M^m, g)\to(N^n,h)$ between Riemannian manifolds is called a {\bf harmonic map} if its tension field $\tau (\varphi)={\rm Trace}_g\nabla{\rm d}\varphi$ vanishes identically.  Harmonic maps include many important  objects studied in mathematics, such as harmonic functions, geodesics, minimal submanifolds, and Riemannian submersions with minimal fibers.  For more on the related theory, applications and interesting links of harmonic maps see \cite{EL78}, \cite{EL83}, \cite{EL88} and \cite{YS97}.\\

\indent A  map  $\varphi: (M^m, g)\to(N^n,h)$ is called a {\bf biharmonic map} if it is a critical point of the bienergy functional
$E_2(\o, g)=\frac{1}{2}\int_{M} |\tau (\varphi)|^2dv_g\,,$
where $\tau (\varphi )$ is the tension field of $\varphi$. The first variation of the bienergy (see \cite{Ji86}) gives the biharmonic map equation as:  
\begin{equation} \label{bihar}
\tau_2(\varphi ) :={\rm Trace}_g\left(  (\nabla^{\varphi})^2 \tau (\varphi ) - R^N( {\rm d}\varphi, \tau (\varphi )){\rm d}\varphi\right) = 0
\end{equation}
where $R^N$ is the Riemannian curvature operator of the manifold $(N, h)$.\\

Clearly, any harmonic map ($\tau (\varphi )\equiv 0$) is biharmonic.  So, we call a  biharmonic map which is  not harmonic a {\bf proper biharmonic map}.

\indent {\bf Biharmonic submanifolds} are the images of biharmonic isometric immersions. As biharmonic maps generalize the concept  of harmonic maps, biharmonic submanifolds are a generalization of minimal submanifolds which are characterized as the images of harmonic isometric immersions. \\

The study of biharmonic submanifolds was initiated by B. Y. Chen \cite{Ch1, CI} and  G. Y. Jiang \cite{Ji86, Ji87} independently. It  has become a hot topic of research with many important progress in the past 30 years. For more detailed account on the origins, the main problems, and  some recent progress  see a recent book \cite{OC} and the vast references therein.\\

Riemannian submersions, as a dual concept of isometric immersions, are  a subclass of conformal submersions which are special cases of horizontally weakly conformal maps.
Recall that a map $\varphi:(M, g)\to (N, h)$ between Riemannian manifolds is  {\bf horizontally weakly conformal} if for any $p \in M$, either  ${\rm d} \varphi_p=0$, or
 ${\rm d} \varphi_{p}|_{ \mathcal{H}_p} : \mathcal{H}_p \to T_{\varphi(p)} N$ is
conformal and onto, where $ \mathcal{H}_p=({\rm ker} {\rm d} \varphi_p)^{\bot}$ is the horizontal subspace of $\varphi$.  This is equivalent to (see, e.g., \cite{BW}) the existence of  function $\lambda\ge 0$, called the {\bf dilation} of $\varphi$, such that $h({\rm d} \varphi
(X),{\rm d} \varphi (Y)) = \lambda^{2}(p)g(X,Y)$ for any  $X, Y \in
 \mathcal{H}_p$.   A {\bf conformal submersion} is horizontally weakly conformal map with dilation $\lambda>0$, and in particular, when $\lambda\equiv 1$ it is called  a {\bf Riemannian submersion}.\\

 The motivations to study biharmonic Riemannian submersions include 
 \begin{itemize}
 \item[(1)] biharmonic Riemannian submersions are a dual concept of biharmonic submanifolds which have exhibited a rich theory and applications in understanding the submanifolds and the ambient spaces; 
 \item[(2)] biharmonic Riemannian submersions are a subclass of
biharmonic conformal submersions which are a generalization of harmonic morphisms. It would be interesting to know what part of the theory, applications and interesting links of harmonic morphisms can be generalized to the case of biharmonic Riemannian submersions; 
\item[(3)] within the class of Riemannian submersions the 4th order biharmonic map equations reduce to a 2nd order PDEs.
\end{itemize}
Biharmonic Riemannian  submersions  were first studied in \cite{On} where the bitension field of a Riemannian submersion was derived under the hypothesis that the tension field  is basic, and  some existence and non existence results were also obtained. Some recent work has  shown that there are many examples of proper biharmonic Riemannian submersions whilst there are also many examples of Riemannian submersions which are neither harmonic nor biharmonic.

\begin{example}
All of the following are {\bf proper biharmonic Riemannian submersions}.\\
(a) $\varphi:M^4=\r^4\setminus\{(0, 0, 0, x): x\in \r\}\to \r^2$ with $\varphi(x_1,\cdots, x_4)=(\sqrt{x_1^2+x_2^2+x_3^2\,}, x_4)$, see \cite{GO} for details.\\
(b)  The projection of the warped product (see \cite{GO} for details)
 \begin{align}\notag
\varphi  : ( \mathbb{R}^2_{+} \times \mathbb{R} , dx^2 +
dy^2+Cy^4 dz^2) &\to (\mathbb{R}^2_{+} ,dx^2 + dy^2) \\\notag
\pi(x,y,z) =(x,y),
\end{align}
 where $\mathbb{R}^2_{+}=\{ (x,y)\in \r^2: y>0\}$  and $C$ be a positive constant. \\
(c) The projection from the product space $H^2\times\r$ (see \cite{WO2} for details) \\ $\varphi :H^2\times\r=(\r^{3},e^{2y}dx^2+dy^2+dz^2) \to
(\r^2 ,dy^2 + dz^2)$, $ \varphi(x,y,z) =(y,z)$.\\
(d) $\varphi:\widetilde{SL}(2,\r)=(\r^{3}_{+},\frac{dx^2+dy^2}{y^2}+(dz+\frac{dx}{y} )^2)
\to (\r^2 ,\frac{dy^2}{y^2} +\frac{dz^2}{2})$ with\\ $\varphi(x,y,z) =(y,z)$, see \cite{WO2} for details.
\end{example}

\begin{example}
The following Riemannian submersions are {\bf neither harmonic nor biharmonic maps}.\\
(e) The metric projection $\varphi:\mathbb{R}^3\to (N^2=\mathbb{R}^3/\mathbb{R},h)$
from Euclidean space $\mathbb{R}^3$ onto the orbit space of  a free $1$-parameter
isometric group action of $\mathbb{R}$. See \cite{WO1} for details. \\
(f) The Riemannian submersion from Nil space (see \cite{WO1} for details)
\begin{align}\notag
\varphi  : ( \mathbb{R}^3 , g_{Nil}={\rm d}x^{2}+{\rm
d}y^{2}+({\rm d}z-x{\rm d}y)^{2}) &\to (\mathbb{R}^2 ,dx^2 + (1+x^2)^{-2}dz^2)
\\\notag
\varphi(x,y,z) =(x,z).
\end{align}
(g)  The Riemannian submersion from Sol space (see \cite{WO4} for details)
 \begin{equation}\notag
\begin{array}{lll}
\varphi:(\r^3,g_{Sol}= e^{2z}{\rm d}x^{2}+e^{-2z}{\rm d}y^{2}+{\rm
 d}z^{2})\to (\r^2, e^{-2z}{\rm d}y^{2}+{\rm
 d}z^{2}),\;\varphi(x,y,z)=(y,z),
\end{array}
\end{equation}
\end{example}

\section{Biharmonic equations of Riemannian submersions}

{\bf I. The case of general fiber dimensions:} It is well known  that  the tension field of a Riemannian submersion $\varphi: (M^m,g)\to (N^n,h)$ is given by
\begin{equation}\label{T}
\tau(\varphi) = -(m-n)d\varphi(\mu),
\end{equation}
where $\mu = \tfrac{1}{m-n} \sum_{s=1}^{m-n} (\nabla^{M}_{e_{s}} e_{s})^{\mathcal{H}}$  with $\{e_s\}$ being an orthonormal
frame of the fibers is the mean curvature vector field of the fibers of the Riemannian submersion. It follows that the harmonicity of a Riemannian submersion is closely related to the minimality of the fibers of the Riemannian submersion.

The biharmonic equation of a  Riemannian submersion  as a special case  of the conformal submersion given in \cite{BFO} and \cite{LO10} seems hard to apply. The following form of the  bitension field of a Riemannian submersion was derived. 
\begin{theorem}\label{MT} $($\cite{AO}$)$
The bitension field of a Riemannian submersion\\ $\varphi: (M^m,g)\to (N^n,h)$ is given by
\begin{eqnarray}\notag
\tau_{2}(\varphi)&=&-(m-n)d\varphi \Big(  \sum_{i=1}^n\big\{\nabla^{M}_{e_{i}}(\nabla^{M}_{e_{i}}\mu)^{\mathcal{H}}- \nabla^{M}_{(\nabla^{M}_{e_i}e_i)^{\mathcal{H}}}\mu+[\mu, (\nabla^{M}_{e_i}e_i)^{\mathcal{V}}]\big\}\\\notag
&&+\sum_{s=n+1}^m\{ [[\mu,e_s],e_s] +[\mu,(\nabla^{M}_{e_s}e_s)^{\mathcal{V}}]\}-(m-n)\nabla^{M}_{\mu}\mu \Big)\\\label{EQ1}
&&-(m-n){\rm Ricci}^N(d\varphi(\mu)),
\end{eqnarray}
where $\{e_{i},e_{s}\} \;(1\le i\le n;  1\le s\le m-n)$  is a local orthonormal frame on $M$ with $e_{i}$ horizontal and $e_{s}$ vertical, and  $\mu$ the mean curvature vector field of the fibers of $\varphi$.
\end{theorem}

\begin{corollary}\label{Co1}$($\cite{AO}$)$
A Riemannian submersion $\varphi: (M^m,g)\to (N^n,h)$  with basic mean curvature vector field $\mu$ of the fibers is biharmonic if and only if
\begin{eqnarray}\label{E1}
&&{\rm d}\varphi \Big(  \sum_{i=1}^n\big\{\nabla^{M}_{e_{i}}(\nabla^{M}_{e_{i}}\mu)^{\mathcal{H}}- \nabla^{M}_{(\nabla^{M}_{e_i}e_i)^{\mathcal{H}}}\mu\big\}
-(m-n)\nabla^{M}_{\mu}\mu \Big)+{\rm Ricci}^N(d\varphi(\mu))=0.
\end{eqnarray}
\end{corollary}

Note that by using (\ref{T}) one can check that Equation (\ref{E1}) is equivalent to the following one which was obtained in  \cite{On}:
\begin{align}\label{Bas}
{\rm Tr}_h\,(\nabla^{N})^2\tau(\varphi)
+\nabla^{N}_{\tau(\varphi)}\tau(\varphi) +{\rm Ricci}^N(\tau(\varphi))=0.
\end{align}

Many examples of proper biharmonic Riemannian submersions come from the projections of a warped product onto its first factor. Here we have

\begin{theorem} \label{MT2}$($\cite{AO}$)$
The Riemannian submersion
\begin{align}\notag
\varphi: ( M^{m} \times N^n , g_M +e^{2\lambda} g_N) \to (M^m , g_M),\;\;
\varphi(x,y) =x
\end{align}
 is biharmonic if and only if
\begin{equation}\label{W-P}
{\rm grad}\, \Delta\lambda+2{\rm Ricci}^{M}({\rm grad}\,\lambda)+\frac{n}{2}{\rm grad}\,|{\rm grad}\lambda|^{2}=0,
\end{equation}
where ${\rm grad}$ and $\Delta$ are the gradient and the Laplace operators on $(M^m, g_M)$.
\end{theorem}

\begin{corollary}$($\cite{AO}$)$
The Riemannian submersion
\begin{align}\notag
\phi: ( M^{m} \times N^n , g_M +e^{2\lambda} g_N) \longrightarrow (M^m , g_M),\;\;
\phi(x,y) =x
\end{align}
from an Einstein Manifold with ${\rm Ricci}^{M}=ag_M$ is biharmonic if and only if
\begin{equation}\label{Einstein}
 \Delta\lambda+2a\lambda+\frac{n}{2}|{\rm grad}\lambda|^{2}=C
\end{equation}
for a constant $C$, where $\Delta$ is the Laplacian on $(M^m, g_M)$.
\end{corollary}

\begin{remark}
(a) It was also proved in \cite{BMO} that the projection
\begin{align}\notag
\phi: ( M^{m} \times N^n , g_M +f^2 g_N) \longrightarrow (M^m , g_M),\;\;
\phi(x,y) =x
\end{align}
 is  proper biharmonic if and only if
\begin{equation}
{\rm Tr}_g\nabla ^2 {\rm grad}\, \ln \,f+{\rm Ricci}^{M}({\rm grad}\, \ln \,f)+\frac{n}{2}{\rm grad}\,|{\rm grad}\, \ln \,f|^2=0.
\end{equation}
(b)  Note that the Riemannian submersion given by the projection of a warped product has basic mean curvature vector field and a biharmonic equation in the form of (\ref{Bas}) was also obtained in \cite{Ur18}. 
\end{remark}

{\bf II. The case of 1-dimensional fibers:}  An important tool to study biharmonic Riemannian submersion $\varphi:( M^{n+1} , g)\to (N^n, h)$ is  the {\bf integrability data} of an adapted frame introduced in \cite{WO1}. For simplicity, we explain the idea in the case of $n=2$.

A local orthonormal frame $\{e_1,\; e_2, \;e_3\}$ is said to be {\bf adapted to
the Riemannian submersion} $\varphi:( M^3 , g)\to (N^2,h)$ if $\{e_1,\; e_2\}$  are basic (i.e.,
they are horizontal and are $\varphi$-related to a local orthonormal frame in the base
space), and $\{ e_3\}$ is vertical. The existence of an adapted frame is easily checked  (cf. e.g., \cite{BW}). For an adapted frame
$\{e_1,\; e_2, \;e_3\}$, we known (see \cite{O})
that $[e_1,e_3]$ and $ [e_2,e_3]$ are vertical and $[e_1,e_2]$ is
$\varphi$-related to $[\varepsilon_1, \varepsilon_2]$ for an orthonormal frame
$\{\varepsilon_1, \varepsilon_2\}$ on the base manifold. It follows that
\begin{equation}\label{RS22}
[\varepsilon_1,\varepsilon_2]=F_1\varepsilon_1+F_2\varepsilon_2,
\end{equation}
for some function $F_1, F_2\in C^{\infty}(N)$. Denoting
$f_i=F_i\circ \varphi, i=1, 2$, we have
\begin{equation}\label{R1}
\begin{cases}
[e_1,e_3]=\kappa_1e_3,\\
[e_2,e_3]=\kappa_2e_3,\\
[e_1,e_2]=f_1 e_1+f_2e_2-2\sigma e_3.
\end{cases}
\end{equation}
where $\kappa_1,\;\kappa_2\;{\rm and}\; \sigma \in C^{\infty}(M)$.
The set of 5 functions $\{ f_1, f_2, \kappa_1,\;\kappa_2\; \sigma\}$
 is called {\bf the integrability data} of the adapted frame of the Riemannian
submersion $\varphi$. Note that by using the integrability data the tensions field of the Riemannian submersion can be expressed as
\begin{align}
\tau( \varphi)=-\kappa_1\epsilon_1-\kappa_2\epsilon_2.
\end{align}
Thus, $\varphi$ is a harmonic Riemannian submersion if and only if $\kappa_1=\kappa_2\equiv 0$. Note also that the integrability data $\sigma\equiv 0$ if and only if the Riemannian submersion has integrable horizontal distribution.

By using the integrability data we can reduce the 4th order PDEs of biharmonic Riemannian submersions into a 2nd order system of PDEs.

\begin{theorem}\label{ID3}\cite{WO1}
A Riemannian submersion $\varphi:(M^3,g)\to (N^2,h)$ 
with an adapted frame $\{e_1,\; e_2, \;e_3\}$ and the integrability
data $ \{ f_1, f_2, \kappa_1,\;\kappa_2\; \sigma\}$ is biharmonic if and only if
\begin{equation}\label{RSB0}
\begin{cases}
-\Delta^{M}\kappa_1-f_1 e_1(\kappa_2)-e_1(\kappa_2 f_1)-f_2
e_2(\kappa_2)-e_2(\kappa_2 f_2)
\\+\kappa_1\kappa_2 f_1 +\kappa_2^2 f_2
+\kappa_1\{-K^{N}+f_{1}^{2}+f_{2}^{2}\}
=0,\\
-\Delta^{M}\kappa_2+f_1 e_1(\kappa_1)+e_1(\kappa_1 f_1)+f_2
e_2(\kappa_1)+e_2(\kappa_1 f_2)\\
-\kappa_1\kappa_2 f_2-\kappa_1^2f_1
+\kappa_2\{-K^{N}+f_{1}^{2}+f_{2}^{2}\} =0,
\end{cases}
\end{equation}
where
$K^{N}=R^{N}_{1212}\circ\varphi=-[e_2(f_1)-e_1(f_2)+f_{1}^{2}+f_{2}^{2}]
$ is the Gauss curvature of $(N^2,h)$.
\end{theorem}

Similarly, a local orthonormal frame $\{e_1,\ldots,e_n, e_{n+1}\}$ is adapted
to the  Riemannian submersion $\varphi:( M^{n+1} , g)\to (N^n, h)$ if $\{e_1,\ldots,e_n\}$  are horizontal and basic, and $e_{n+1}$ is vertical. It follows from
\cite{O} that $[e_i,e_{n+1}]$ $(i=1, 2, \ldots, n)$ are vertical
and $[e_i,e_j]$ $(i, j=1, 2, \ldots, n)$ are $\varphi$-related to
$[\varepsilon_i, \varepsilon_j]$ for an orthonormal frame $\{\varepsilon_i, \ldots,
\varepsilon_n\}$ on $(N^n, h)$. By setting
\begin{equation}\label{RS22}
[\varepsilon_i,\varepsilon_j]=F_{ij}^k\varepsilon_k,
\end{equation}
for $F^k_{ij}\in C^{\infty}(N)$ and introducing
$f^k_{ij}=F^k_{ij}\circ \varphi,\;\forall\; i, j, k=1, 2, \ldots, n$, we have
\begin{equation}\label{LB}
\begin{cases}
[e_i,e_{n+1}]=\kappa_ie_{n+1},\\
[e_i,e_j]=f^k_{ij} e_k-2\sigma_{ij}e_{n+1},\;\;\; i, j=1, 2, \ldots,
n,
\end{cases}
\end{equation}
where $\kappa_i,\,{\rm and}\; \sigma_{ij} \in C^{\infty}(M)$. The set of functions
 $ \{f^k_{ij},  \kappa_i,\,\sigma_{ij}\}$ is called {\bf the
integrability data} of the adapted frame of the Riemannian
submersion $\varphi$. By (\ref{LB}) one can easily check that
\begin{eqnarray}
f_{ij}^k=-f_{ji}^k\; \;{\rm and}\;\sigma_{ij}=-\sigma_{ji},\;\;\forall\; i,j, k=1,2, \cdots, n.
\end{eqnarray}

Note that the integrability data of a Riemannian submersion $\varphi:( M^{n+1} , g)\to (N^n, h)$ contains the following geometric information of the spaces and the harmonicity of the Riemannian submersion.

\begin{corollary}
For a Riemannian submersion $\varphi:(M^{n+1},g)\to (N^n,h)$  with an adapted frame $\{e_1,\ldots, e_{n+1}\}$ and the assciated
integrability data $ \{f^k_{ij}, \kappa_i, \sigma_{ij}\}$ defined on an open set $U\subseteq M$, we have\\
(1) If $f^k_{ij}\equiv 0,\;\forall\; i,j, k$, then $\varphi (U)$ is locally flat;\\
(2)  $ \sigma_{ij}\equiv 0,\;\forall\; i,j$ if and only if the horizontal distribution on $U$ is integrable;\\
(3) $ \kappa_i\equiv 0\;\forall\; i$ if and only if the Riemannian submersion is harmonic.
\end{corollary}

\begin{proof}
For (1), by definition, $f^k_{ij}=F^k_{ij}\circ \varphi$. It follows that if $f^k_{ij}\equiv 0\;\forall\; i,j, k$, then $[\varepsilon_i,\varepsilon_j]=F_{ij}^k\varepsilon_k\equiv 0$ for an orthonormal frame $\{\varepsilon_j\}$ on $\varphi(U)$. It is easily check that in this case the Riemann curvature tensor vanishes identically on $\varphi(U)$ and hence a classical result in differential geometry implies that $\varphi(U)$ is locally flat. For (2), recall that the horizontal distribution of a Riemannian submersion in integrable if and only if its A tensor  $A_X Y=\frac{1}{2}[X, Y]^{\mathcal{V}}$ vanishes identically for any horizontal vector fields $X, Y$ (see \cite{O}). This and (\ref{LB}) gives (2). Finally, (3) follows from the tension field formula
\begin{align}
\tau( \varphi)=-\sum_{i=1}^n\kappa_i\epsilon_i.
\end{align}
\end{proof}

Similarly, the biharmonicity of a Riemannian submersion can also be described by the integrability data as

\begin{theorem}\cite{AO}\label{1D}
A Riemannian submersion $\varphi:(M^{n+1},g)\to (N^n,h)$  with an adapted frame $\{e_1,\ldots, e_{n+1}\}$ and the
integrability data $ \{f^k_{ij}, \kappa_i, \sigma_{ij}\}$ is biharmonic if and only if
\begin{eqnarray}\label{1DE}
&&\Delta\kappa_k+\sum^{n}_{i,j=1}\Big(2e_i(\kappa_j)P^{k}_{ij}+\kappa_j\big[ e_i (P^{k}_{ij})+P^{l}_{ij}P_{il}^k -\kappa_iP_{ij}^k - P_{ii}^lP_{lj}^k\big] \Big)\\\notag
&&+{\rm Ricci}^N(d\varphi(\mu), d\varphi(e_k))=0, \;\;k=1,2,\ldots, n,
\end{eqnarray}
where $P_{ij}^k=\frac{1}{2}(-f_{ik}^j-f_{jk}^i +f_{ij}^k)$ for all $ i,j, k=1,2, \cdots, n$ and $\Delta$ denotes the Laplacian of $(M^{n+1}, g)$.
\end{theorem}

\section{Some recent work on biharmonic Riemannian submersions}

An important problem in the study of biharmonic submanifolds is to classify biharmonic submanifolds in some well known model spaces like space forms or more general homogeneous spaces.  For example, let $M^3(c)= \r^3, H^3, S^3$ denote  the {\bf 3-dimensional space forms}, i.e.,  Euclidean, hyperbolic, or spherical space of constant sectional curvature $K=0, -1, 1$ respectively, then the following classification is well known.\\
{\bf Theorem A.} \cite{CI, Ji87, CMO} $\varphi: (N^2, h)\to M^3(c)$ is a proper biharmonic isometric immersion if and only if, up to a homothety, $\varphi$ is the inclusion $i:S^2\to S^3, i(x)=(x, 1)/\sqrt{2}$ or its restriction to an open set of $S^2$.\\

In this section, we review the recent work \cite{WO1, WO2, WO3, WO4, WO5} which give some classifications of proper biharmonic Riemannian submersions from 3-dimensional manifolds including Thurston's  eight models  of 3-dimensional geometries, BCV spaces, and the product manifold $M^2\times \r$. 

Recall that {\bf a BCV space} is a 3-dimensional Bianchi-Cartan-Vranceanu space defined by the 2-parameter family of Riemannian manifolds
$$M^3_{m,\;l}=\left(\r^{3},g=\frac{dx^2+dy^2}{[1+m(x^2+y^2)]^2}+[dz+\frac{l}{2}\frac{y
dx-x dy}{1+m(x^2+y^2)}]^2\right).$$
 With different choices of the constants $m, l$, this family gives  7  well-known model spaces $\r^3,\; S^3,\; S^2\times\r,\; H^2\times\r,\; \widetilde{SL}(2,\r), \;Nil, \; SU(2)$. \\
 
Recall also that {\bf Thurston's  8 models of 3-dimensional geometries} refers to: $\r^3,\;H^3,\; S^3,\; S^2\times\r,\; H^2\times\r,\; \widetilde{SL}(2,\r), \;Nil, \; Sol$. So the family of BCV spaces  contains six of  the Thurston's eight models for 3-dimensional geometries except the hyperbolic space $H^3$ and the Sol space.
 \\

\noindent {\bf 3.1 Biharmonic Riemannian submersions from space forms}\\
The first classification of proper biharmonic Riemannian submersions shows  a contrast to that of proper biharmonic isometric immersions given in Theorem A.
\begin{theorem}\cite{WO1}\label{Th0}
There is NO proper biharmonic Riemannian submersion\\ $\varphi:M^3(c)\to (N^2,h)$ 
from a space form no matter what the target surface is.
\end{theorem}
\noindent {\bf 3.2 Biharmonic Riemannian submersions from  BCV spaces}

\begin{theorem}\label{Th1}\cite{WO2}
A proper biharmonic Riemannian submersion $\pi:M^3_{m,\;l}\to (N^2,h)$
from a BCV 3-space  exists only  in the cases: $H^2(4m)\times\r
\to \r^2$,\; or\\ $\widetilde{SL}(2,\r)\to\r^2$.
\end{theorem}
\begin{corollary}\cite{WO2}
There is NO proper biharmonic Riemannian submersion from
$\r^3,\; S^3,\; S^2\times\r,\;Nil, \; SU(2)$ onto a surface
\end{corollary}

Note that Theorem \ref{Th1} tells us that among all 7 model spaces in the BCV space family only two: $H\times \r$ and $\widetilde{SL}(2,\r)$ admit proper Riemannian submersions onto a surface which has to be Euclidean plane $\r^2$. In each of these cases many examples of proper biharmonic Riemannian submersions are also constructed in \cite{WO2}. The following uniqueness theorem was proved later in \cite{WO3}.

\begin{proposition}\label{Pr3}\cite{WO3}
A Riemannian submersion  $\pi:H^2\times\r\to (N^2,h)$ is proper biharmonic if and only if  $(N^2,h)$ is flat, and, up to an isometry, the map can be expressed as $\varphi: H^2\times\r\to\r^2$ with
$\varphi:(\r^{3},e^{2y}dx^2+dy^2+dz^2)\to (\r^2 ,dy^2 + dz^2),\; \varphi(x,y,z) =(y,z)$, a projection of a warped product.
\end{proposition}

\noindent {\bf 3.3 Biharmonic Riemannian submersions from Thurston's 3-dimensional geometries}

It follows from Theorem \ref{Th0}, Theorem \ref{Th1},  and the relationship between the family of BCV spaces and Thurston's 8 models of 3-dimensional geometries, the classification of  proper biharmonic Riemannian submersions from Sol space will completes the classification of all 8 model spaces of Thurston's 3-dimensional geometries. This is done by the following theorem.

\begin{theorem}\cite{WO4}
There is no proper biharmonic Riemannian submersion\\ $\varphi: Sol\equiv (\r^3,g_{Sol}= e^{2z}{\rm d}x^{2}+e^{-2z}{\rm d}y^{2}+{\rm
 d}z^{2})\to (N^2, h)$ from Sol space into a surface.
\end{theorem}

Thus, we can summarize the classification of proper biharmonic Riemannian submersions from Thurston's 3-dimensional geometries as follows.

\begin{theorem}\cite{WO1, WO2, WO4}
A proper biharmonic Riemannian submersion from Thurston's 3-dimensional geometries: $\r^3,\;H^3,\; S^3,\; S^2\times\r,\; H^2\times\r,\; \widetilde{SL}(2,\r), \;Nil, \; Sol$
  exists only  in the cases: $H^2(4m)\times\r \to \r^2$,\; or $\widetilde{SL}(2,\r)\to\r^2$.
\end{theorem}

Again, we remark that the biharmonicity of Riemannian submersions from 3-dimensional geometries not just imposes conditions the geometry and/or topology of the  the total spaces but also forces the based space to be flat. \\

\noindent {\bf 3.4 Biharmonic Riemannian submersions from $M^2\times \r$}

For biharmonic Riemannian submersions from the product $M^2\times \r$, where $(M^2, g)$ is a general 2-dimensional manifold, we have

\begin{theorem}\label{Ths1}\cite{WO3}
For a  proper biharmonic Riemannian submersion\\ $\pi:M^2\times\r\to (N^2,h)$  from the  product space, we have either\\
(i) $(N^2,h)$ is flat, and  locally, up to an isometry of the domain and/or codomain, $\pi$ is the projection of the special twisted product
\begin{equation}\label{TP}
 \pi: (\r^3, e^{2p(x, y)}dx^2+dy^2+dz^2) \to(\r^2,dy^2+dz^2),\;\pi(x, y, z)=(y,z),
 \end{equation}
with the twisting function $p(x, y)$  such that $p_y\neq0$ and  solves the PDE \\
\begin{equation}\label{mr}
\begin{array}{lll}
\Delta p_y :=p_{yyy}+p_{yy}p_y+e^{-2p(x, y)}(p_{xxy}-p_{xy}p_x)=0, \; {rm or}
\end{array}
\end{equation}

(ii) $(N^2,h)$ is non-flat, and  locally, up to an isometry of the domain and/or codomain, the map can be  expressed as
\begin{equation}\label{th00}
\begin{array}{lll}
\pi:(\r^3,e^{2p(x,y)}dx^2+dy^2+dz^2)\to (\r^2, dy^2+e^{2\lambda(y,\phi)}d\phi^2),\\
\pi(x, y, z)=(y,F\left(z- \int e^{\varphi(x)}dx\right)),
\end{array}
\end{equation}
where $p(x,y)=\ln|\tan\alpha(y)|+\varphi(x)$, $\lambda=\ln |\sin \alpha (y)|+w(\phi)$ with the functions $\varphi(x)$, $w(\phi)$  and nonconstant function $F(u)$ satisfying  $F'(z- \int e^{\varphi(x)}dx)=e^{-w(\phi)}$ and $z- \int e^{\varphi(x)}dx=\int e^{w(\phi)}d\phi$,  and $\alpha(y)$ is the angle between the fibers of $\pi$ and $ E_3=\partial_z$ solving the ODE \begin{equation}\label{zr3}\notag
\begin{array}{lll}
\alpha'''\sin\alpha\cos^2\alpha+\cos\alpha(\sin^2\alpha+3)\alpha'\alpha''+\sin\alpha(2\cos^2\alpha+3)\alpha'^3=0.
\end{array}
\end{equation}
\end{theorem}

Note that the Riemannian submersion in Case (ii) of Theorem \ref{Ths1} is a map (\ref{th00}) depending on the function $F$, the following corollary shows that after changes of coordinates,  a proper biharmonic Riemannian submersions from a product manifold onto a non-flat surface  can be described explicitly as a simple map  between two special warped product manifolds with the warping functions solving an ODE.

\begin{corollary}\label{co1}\cite{WO3}
A proper biharmonic Riemannian submersion $\pi:M^2\times\r\to (N^2,h)$ from product manifold into a non-flat surface is locally, up to an isometry of the domain and/or codomain, $\pi$ is a map between two special warped product spaces  given by
\begin{equation}\label{coo1}
\begin{array}{lll}
\pi:(\r^3, \tan ^2\alpha(y)\,d t^2+dy^2+dz^2)\to (\r^2,dy^2+\sin^2 \alpha(y)\,d\psi^2)\\
\pi(t,y,z)=(y,z- t),
\end{array}
\end{equation}
where $\alpha(y)$ is the angle between the fibers of  $\pi$ and $ E_3=\frac{\partial}{\partial z}$ solving the ODE
\begin{equation}\label{cx}
\begin{array}{lll}
\alpha'''\sin\alpha\cos^2\alpha+\cos\alpha(\sin^2\alpha+3)\alpha'\alpha''+\sin\alpha(2\cos^2\alpha+3)\alpha'^3=0.
\end{array}
\end{equation}
\end{corollary}



\begin{thebibliography}{199}


\bibitem{AO} M. Akyol and Y. -L. Ou, {\em Biharmonic Riemannian submersions}, Annali di Mate. Pura ed Appl., 2019, https://doi.org/10.1007/s10231-018-0789-x. 
\bibitem{BFO} P. Baird, A. Fardoun and S. Ouakkas, {\em Conformal and semi-conformal biharmonic maps}, Ann. Glob. Anal. Geom., 34 (2008), 403-414.


\bibitem{BW} P. Baird and J. C. Wood, {\em Harmonic morphisms between Riemannian manifolds}, London Math. Soc. Monogr. (N.S.) No. 29, Oxford Univ. Press, 2003. 



\bibitem{BMO} A.  Balmu\c s, S. Montaldo, C. and Oniciuc, {\em Biharmonic maps between warped product manifolds},  J. Geom. Phys., 57. (2007), 449--466.

\bibitem{CMO} R. Caddeo, S. Montaldo,  and  C. Oniciuc,  {\em Biharmonic submanifolds of $S^3$}, Internat. J. Math., 12(2001),  867--876.

\bibitem{Ch1} B. Y. Chen, {\em Some open problems and conjectures on submanifolds of finite type}, Soochow J. Math. 17(2) (1991), 169--188.
\bibitem{CI} B. Y. Chen and  S. Ishikawa, {\em Biharmonic surfaces in pseudo-Euclidean spaces},  Memoirs Fac. Sci. Kyushu Univ. Ser. A, Math., 45(1991), 323--347.

\bibitem{EL78} J. Eells and L. Lemaire, {\em A report on harmonic maps}, Bull. London Math. Soc. 10 (1)  (1978), 1--68.
\bibitem{EL83} J. Eells and L. Lemaire, {\em Selected topics in harmonic maps}, CBMS, 50, Amer. Math. Soc. (1983).
\bibitem{EL88} J. Eells and L. Lemaire, {\em Another report on harmonic maps}, Bull. London Math. Soc.  20 (5) (1988), 385--524.


\bibitem{GO} E. Ghandour and Y. -L. Ou, {\em Generalized harmonic morphisms and horizontally weakly conformal biharmonic maps}, J. Math. Anal. Appl., 464 (2018), 924-938.


\bibitem{Ji86} G. Y. Jiang, {\em $2$-Harmonic maps and their first and second variational formulas}, Chin. Ann. Math. Ser. A,  7(1986) 389-402.
\bibitem{Ji87} G. Y. Jiang, {\em Some non-existence theorems of $2$-harmonic isometric immersions into Euclidean spaces}, Chin. Ann. Math. Ser. A,  8 (1987) 376-383.


\bibitem{LO10} E. Loubeau and Y. -L. Ou, {\em Biharmonic maps and morphisms from conformal mappings}, Tohoku Math J., 62 (1), (2010), 55-73.

\bibitem{O} B. O'Neill, {\em The fundamental equations of a submersion}, Michigan Math. J. 13 (1966), 459--469.
\bibitem{On} C. Oniciuc, {\em Biharmonic maps between Riemannian manifolds}, An. Stiint. Univ. Al. I. Cuza Iasi Mat (N.S.) 48 (2002), 237--248.
\bibitem{On2} C. Oniciuc, {\em Biharmonic submanifolds in space forms}, Habilitation Thesis, (2012), www.researchgate.net, https://doi.org/10.13140/2.1.4980.5605.



\bibitem{OC}  Y. -L. Ou and B. -Y. Chen, {\em Biharmonic submanifolds and biharmonic maps in Riemannian geometry}, World Scuentific,  2020.
\bibitem{Ur18} H. Urakawa, {\em Harmonic maps and biharmonic maps on principal bundles and warped products},  J. Korean Math. Soc., 55(3), (2018), 553-574.
\bibitem{Ur19} H. Urakawa, {\em Harmonic maps and biharmonic Riemannian submersions}, Note di Mate. 39 (1) (2019), 1--23.
\bibitem{WO1} Z. -P. Wang and   Y. -L. Ou,  {\em Biharmonic Riemannian submersions from 3-manifolds},  Math.  Zeitschrift, 269 (3) (2011), 917-925.
\bibitem{WO2} Z. -P. Wang and   Y. -L. Ou,   {\em Biharmonic Riemannian submersions from a 3-dimensional BCV space}, J. Geom. Anal., 2024 to appear.
\bibitem{WO3} Z. -P. Wang and   Y. -L. Ou,   {\em Biharmonic isometric immersions into and biharmonic Riemannian submersions from $M^2\times \r$}, preprint 2023, arXiv:2302.11545.
\bibitem{WO4} Z. -P. Wang, Y. -L. Ou, and Q.-L. Liu  {\em Harmonic and biharmonic Riemannain submersions from Sol space}, preprint 2023, arXiv:2302.11693.
\bibitem{WO5} Z. -P. Wang and   Y. -L. Ou,   {\em Biharmonic isometric immersions into and biharmonic Riemannian submersions from a generalized Berger sphere }, preprint 2023, arXiv:2302.11692.
\bibitem{YS97} S. T. Yau and R. Schoen, {\em Lectures on harmonic maps}, International Press Incorporated, Boston 385 Someville Ave, Someville, MA, U.S.A, 1997.

\end{thebibliography}
\end{document}